\documentclass[12pt]{article}
\usepackage{amssymb,amsmath,amsopn}
\usepackage{mathtools}
\usepackage[dvips]{graphicx}   
\usepackage{color,epsfig}      
\usepackage{url}

\makeatletter
\def\@settitle{\begin{center}%
  \baselineskip14\p@\relax
  \bfseries
  \uppercasenonmath\@title
  \@title
  \ifx\@subtitle\@empty\else
     \\[1ex]\uppercasenonmath\@subtitle
     \footnotesize\mdseries\@subtitle
  \fi
  \end{center}%
}
\def\subtitle#1{\gdef\@subtitle{#1}}
\def\@subtitle{}
\makeatother

\newtheorem{cor}{Corollary}

\newtheorem{thm}{Theorem}
\newtheorem{prop}{Proposition}
\newtheorem{lem}{Lemma}

\newtheorem{defn}{Definition}

\newtheorem{conj}{Conjecture}
\newtheorem{rem}{Remark}

\DeclareMathOperator{\conv}{conv}

\DeclareMathOperator{\vol}{vol}
\DeclareMathOperator{\area}{area}
\DeclareMathOperator{\perim}{perim}

\DeclareMathOperator{\inter}{int}

\DeclareMathOperator{\bd}{bd}

\DeclareMathOperator{\proj}{proj}

\newcommand{\xx}{\mathbf{x}}

\newcommand{\K}{\mathcal{K}}

\newcommand{\KK}{\mathbf K}

\renewcommand{\Re}{\mathbb{R}}

\newcommand{\Z}{\mathbb{Z}}
\newcommand{\Eu}{\mathbb{E}}

\newcommand{\Sph}{\mathbb{S}}
\newcommand{\Sd}{\mathbb{S}^{d-1}}

\newcommand{\Ee}{{\mathbb E}}
\newcommand{\Ed}{\Ee^d}

\newcommand{\F}{\mathcal{F}}

\newcommand{\C}{\mathbf{C}}

\newcommand{\e}{\mathbf{e}}
\newcommand{\p}{\mathbf{p}}

\renewcommand{\o}{\mathbf{o}}

\newcommand{\y}{\mathbf{y}}

\title{Non-separable arrangements of convex bodies revisited
\footnote{\hspace*{.35cm} 2010 Mathematics Subject Classification: 52C17, 52C15, 52A40}}
\author{K\'{a}roly Bezdek\thanks{Partially supported by a Natural Sciences and 
Engineering Research Council of Canada Discovery Grant.} and Zsolt L\'angi\thanks{Partially supported by the National Research, Development and Innovation Office, NKFI, K-147544, the ERC Advanced Grant ``ERMiD'', Projects no. TKP2021-NVA-09 and 2024-1.2.8-T\'ET-IPARI-CN-2025-00011, with the
support provided by the Ministry of Innovation and Technology of Hungary from the National Research, Development and Innovation Fund.}
}

\begin{document}

\maketitle

\begin{abstract}
A problem posed by Erd\H os in 1945 initiated the study of non-separable arrangements of convex bodies. A finite collection of convex bodies in Euclidean $d$-space is called a non-separable family (or NS-family) if every hyperplane intersecting their convex hull also intersects at least one member of the family. Recent work has focused on minimal coverings of NS-families consisting of positive homothetic convex bodies. In this paper, we strengthen these results by establishing their stability analogues and analyze weakly non-separable families of convex polytopes. We also examine weakly $k$-impassable families of convex $d$-polytopes for $0<k<d-1$ and study tightness of NS-lattice arrangements. 
\end{abstract}

\section{Non-separable arrangements - a brief overview}

Goodman and Goodman \cite{GG45} proved an elementary conjecture of Erd\H os: If $n>1$ circles in the plane have radii $r_1, r_2, \dots , r_n$ and no line separates them, then they can all be covered by a circle of radius $\sum_{i=1}^n r_i$. Shortly after the appearance of \cite{GG45}, Hadwiger \cite{Had} broadened the question and initiated a line of research that has since been pursued by others. In what follows, we recall some basic definitions and key concepts from the authors' paper \cite{BezLan16}. Let $\Ee^d$ denote the $d$-dimensional Euclidean space with the standard Euclidean norm $\|\cdot\|$. Let $X+Y$ denote the Minkowski sum or simply, vector sum of the sets $X, Y\subseteq \Ee^d$. If $\KK$ is a convex body, i.e., a compact convex set with nonempty interior in $\Ee^d$, then let $\K := \{ \xx_i + \tau_i \KK\ |\  \xx_i\in \Ee^d, \tau_i>0, i=1,2,\ldots, n\}$, where $d\ge 2$, $n\ge 2$ and $\xx_i + \tau_i \KK$ is called a positive homothetic copy of $\KK$ with homothety coefficient $\tau_i>0$. Assume that $\K$ is a {\it non-separable family}, in short, an {\it NS-family}, meaning that every hyperplane intersecting $\conv \left( \bigcup \K \right)$ intersects a member of $\K$ in $\Ee^d$, i.e., there is no hyperplane that strictly separates some elements of $\K$ from all the other elements of $\K$ in $\Ee^d$. Then, let $\lambda(\K) > 0$ denote the smallest positive value $\lambda$ such that a translate of
$\lambda \left( \sum_{i=1}^n \tau_i \right) \KK$ covers $\bigcup \K$. As in \cite{BezLan16}, for $d \geq 2$ let $\lambda_{\rm simplex}^{(d)}$ denote the supremum of $\lambda(\K)$, where $\K$ runs over the NS-families of finitely many positive homothetic $d$-simplices in $\Ee^d$. It was proved in \cite{BezLan16} that $\lambda_{\rm simplex}^{(d)}$ is a non-decreasing sequence of $d$. Moreover, $\lambda_{\rm simplex}^{(d)} \geq \lambda_{\rm simplex}^{(2)}\geq 2/3+2/(3\sqrt{3})= 1.0515\ldots>1$ for all $d\ge 3$. The special role of $\lambda_{\rm simplex}^{(d)}$ is highlighted by the following statement proved by the authors in \cite{BezLan16}.

\begin{thm}\label{simplex bound}
For all $d\ge 2$, $\lambda_{\rm simplex}^{(d)}=\sup_{\K} \lambda(\K)$, where $\K$ runs over the NS-families of finitely many positive homothetic copies of an arbitrary convex body $\KK$ in $\Ee^d$.
\end{thm}

On the other hand, improving the upper bound $\lambda_{\rm simplex}^{(d)}\leq d$ from \cite{BezLan16}, Akopyan, Balitskiy, and Grigorev \cite{ABG} have derived the following stronger estimate. 

\begin{thm}[Akopyan, Balitskiy, and Grigorev, 2018]\label{ABG-bound}
$\lambda_{\rm simplex}^{(d)}\leq (d+1)/2$ for all $d\geq 2$.
\end{thm}
Unfortunately, the above estimate does not seem to be sharp even for $d=2$. In fact, the authors asked in \cite{BezLan16} whether $\lambda_{\rm simplex}^{(2)}=2/3+2/(3\sqrt{3})= 1.0515\ldots$ holds. 



However, if the convex body $\KK$ is centrally symmetric, then one can expect much stronger results. Here we recall that a set $S \subset \Ee^d$ is centrally symmetric with center $\mathbf{q} \in \Ee^d$ if $S= 2\mathbf{q}-S$, and $\mathbf{o}$-symmetric if $S=-S$, where $\mathbf{o}$ denotes the origin of $\Ee^d$. The following result was proved by the authors in \cite{BezLan16}.

\begin{thm}\label{centrally symmetric convex bodies}
For all $d\ge 2$ and $n\ge 2$, and for every $\mathbf{o}$-symmetric convex body $\KK_0$ and every NS-family $\K := \{ \xx_i + \tau_i \KK_0\ |\  \xx_i\in \Ee^d, \tau_i>0, i=1,2,\ldots, n\}$ the inequality $\lambda(\K) \leq 1$ holds.
\end{thm}

Next, recall the following definition from \cite{BezLan16}. Let $\K := \{ \xx_i + \tau_i \KK \ |\  \xx_i\in \Ee^d, \tau_i>0, i=1,2,\ldots, n\}$ be a family of positive homothetic copies of the convex body $\KK$ in $\Ee^d$ and let $0\le k\le d-1$. We say that $\K$ is a {\it $k$-impassable arrangement}, in short, a {\it $k$-IP-family} if every $k$-dimensional affine subspace of $\Ee^d$ intersecting $\conv \left( \bigcup \K \right)$ intersects an element of $\K$. Let $\lambda_k(\K) > 0$ denote the smallest positive value $\lambda$ such that some translate of $\lambda \left( \sum\limits_{i=1}^n \tau_i \right) \KK$ covers $\bigcup \K$, where $\K$ is an arbitrary $k$-IP-family. A $(d-1)$-IP-family is simply called an NS-family and in that case $\lambda_{d-1}(\K)=\lambda(\K)$. In order to state the main result of  \cite{BezLan16} on $k$-IP-families of positive homothetic convex bodies in $\Ee^d$ we need the following definitions and statement from Section 3.2 of \cite{Sch14}. Let $\KK_1'$ and $\KK_2'$ be nonempty, compact, convex sets in $\Ee^d, d\ge 2$. We say that $\KK_2'$ {\it slides freely inside} $\KK_1'$ if to each boundary point $\xx$ of $\KK_1'$ there exists a translation vector $\y\in\Ee^d$ such that $\xx\in\y+\KK_2'\subseteq\KK_1'$.
Furthermore, we say that $\KK_2'$ is a {\it summand} of $\KK_1'$ if there exists a nonempty, compact, convex set $\KK'$ in $\Ee^d$ such that $\KK_2'+\KK'=\KK_1'$. It is not hard to see that $\KK_2'$ is a summand of $\KK_1'$ if and only if $\KK_2'$ slides freely inside $\KK_1'$ (see \cite[Theorem 3.2.2.]{Sch14}). Now, the main result of  \cite{BezLan16} on $k$-IP-families reads as follows. 

\begin{thm}\label{k-IP-family}
Let $\KK$ be a $d$-dimensional convex body and $\K:=\{ \KK_i := \xx_i + \tau_i \KK\ |\  \xx_i\in \Ee^d, \tau_i>0, i=1,2,\ldots, n\}$
be a $k$-IP family of positive homothetic copies of $\KK$ in $\Ee^d$, where $0 \leq k \leq d-2$. Then $\conv \left(\bigcup \K \right)$ slides freely in $\left( \sum_{i=1}^n \tau_i \right) \KK$ (i.e., $\conv \left(\bigcup \K \right)$ is a summand of $\left( \sum_{i=1}^n \tau_i \right) \KK$) and therefore $\lambda_k(\K) \leq 1$. 
\end{thm}

It is natural to say that an arrangement of (finitely many or infinitely many) convex bodies is {\it separable} in $\Ee^d$ if there exists a hyperplane that strictly separates some members of the arrangement from all the other members of the arrangement in $\Ee^d$. If an arrangement of convex bodies is not separable in $\Ee^d$, then we say that it is a {\it non-separable arrangement}, in short, an {\it NS-arrangement}. Let $\mathcal{L}^d$ denote the set of full-dimensional lattices in $\Ee^d$, i.e., discrete subgroups of $\Ee^d$ of full rank. Every lattice $L\in\mathcal{L}^d$ can be written as $L= A\Z^d$ for some invertible matrix $A\in {\rm GL}_d(\Re)$. Given a subset $S\subseteq\Ee^d$ and a lattice $L\in\mathcal{L}^d$, we call $S+L:=\cup_{\mathbf{z}\in L}(S+\mathbf{z})$ a {\it lattice arrangement} of $S$ in $\Ee^d$. If $S+L$ is an NS-arrangement, the we call it an {\it NS-lattice arrangement}. G. Fejes T\'oth (\cite{FeMa74}) asked the fundamental question on finding the smallest density of NS-lattice arrangements of balls in $\Ee^d$. More generally, L. Fejes T\'oth and E. Makai Jr. \cite{FeMa74} asked for the infimum of the densities of thinnest NS-lattice arrangements of convex bodies in $\Ee^d$. They proved the following theorem in \cite{FeMa74}. 
\begin{thm}[L. Fejes T\'oth and Makai Jr., 1974]\label{FTL-Makai}
The density of any NS-lattice arrangement of an arbitrary convex domain is at least $\frac{3}{8}$ in $\Ee^2$. Equality holds only if the convex domain is a triangle spanned by one vertex and the midpoints of the opposite sides of a fundamental parallelogram of the corresponding NS-lattice arrangement.
\end{thm}
In 1978, Makai Jr. \cite{Ma78} established a remarkable connection between the volume product of a convex body, its maximal lattice packing density and the minimal density of an NS-lattice arrangement of its polar body. Consequently, he formulated the following conjecture. If $\mathbf{K}$ be a convex body in $\Ee^d$, then $V_{d}(\mathbf{K})$ denotes the $d$-dimensional volume (i.e., Lebesgue measure) of $\mathbf{K}$. Moreover, for $L\in\mathcal{L}^d$ we denote the determinant of $L$, i.e., the $d$-dimensional volume of a fundamental domain of $L$, by ${\rm det}(L)$.
\begin{conj}[Makai Jr., 1978]\label{Makai}
Let $\mathbf{K}$ be a convex body in $\Ee^d$, $d>1$ and let $L\in\mathcal{L}^d$ such that $\mathbf{K}+L$ is an NS-lattice arrangement. Then $\frac{V_d(\mathbf{K})}{{\rm det}(L)}\geq \frac{d+1}{2^dd!}$.
\end{conj}
Note that $\frac{V_d(\mathbf{K})}{{\rm det}(L)}$ is the density of the NS-lattice arrangement $\mathbf{K}+L$ in $\Ee^d$ (i.e., the fraction of space covered by $L$ translates of $\mathbf{K}$) and so, Theorem~\ref{FTL-Makai} proves Conjecture~\ref{Makai} for $d=2$. (For more details see \cite{FeMa74} and \cite{Ma78}.) On the other hand, Conjecture~\ref{Makai} is open for all $d>2$. For several partial results and further connections to Mahler's conjecture on volume product and covering minima of Kannan and Lov\'asz \cite{KaLo} we refer the interested reader to \cite{GoSc}.

\section{New Results}\label{new results}

\subsection{A stability analogue of Theorem~\ref{centrally symmetric convex bodies}}
 
If the positive homothetic convex bodies in Theorem~\ref{centrally symmetric convex bodies} have pairwise disjoint interiors and their centers lie on a common line, with consecutive bodies along the line touching one another, then the estimate of Theorem~\ref{centrally symmetric convex bodies} is sharp. This observation naturally motivates the search for a corresponding stability statement. It will be useful to have the following concept. Let $\KK$ be an $\o$-symmetric convex body in $\Ee^d$ and let $X$ be a compact set in $\Ee^d$. Then the $\KK$-circumradius of $X$ is the smallest value of $\lambda>0$ such that a translate of $\lambda \KK$ contains $X$.

\begin{thm}\label{NS-stability}
Let $\KK$ be an $\o$-symmetric convex body in $\Ee^d$ with $C^2$-class boundary and strictly positive Gaussian curvature everywhere. Let $\tau_1, \tau_2, \ldots, \tau_n >0$. Then there are some positive constants $C:=C(\tau_1,\tau_2,\ldots,\tau_n, \KK)$ and $D:=D(\tau_1,\tau_2,\ldots,\tau_n, \KK)$ with the following property. If $\mathcal{F}:= \{ \p_i + \tau_i \KK \ |\  \p_i\in \Ee^d, \tau_i>0, i=1,2,\ldots, n\}$ is an NS-family whose $\KK$-circumradius is at least $\left(\sum_{i=1}^n \tau_i\right) - \varepsilon$ for some $\varepsilon$ with $0<\varepsilon \leq D$, then there is a straight line $L$ whose distance from every $\p_i$ is at most $C \varepsilon^{1/2}$. Furthermore, the exponent $1/2$ is the best possible for any fixed values of the $\tau_i>0, i=1,2,\ldots, n$. 
\end{thm}

\begin{rem}
We note that if we consider a centrally symmetric convex body $\mathbf{K}$ in $\Ee^d$ which is not strictly convex, then one cannot expect the stability result of Theorem~\ref{NS-stability} to hold for homothetic copies of $\mathbf{K}$. For example, there are $n>1$ translates of a unit $d$-cube in $\Ee^d$ such that the smallest homothetic cube containing them is of edge length $n$ and still the centers of the unit cubes do not lie on a line in $\Ee^d$.
\end{rem}

\subsection{On coverings of weakly non-separable families of convex polytopes}

To introduce our new result, we require a notion that was originally defined for simplices by Akopyan, Balitskiy, and Grigorev in \cite{ABG}. We extend this concept to general convex polytopes as follows.

\begin{defn}\label{erdos extended}
Let $\mathbf{P}$ be a $d$-dimensional convex polytope in $\Ee^d$ and let $\mathcal{P} := \{ \xx_i + \tau_i \mathbf{P}\ |\  \xx_i\in \Ee^d, \tau_i>0, i=1,2,\ldots, n\}$, where $d\ge 2$ and $n\ge 2$.
We call $\mathcal{P}$ a {\rm weakly non-separable family}, in short, a {\rm WNS-family}, if every hyperplane parallel to some facet of $\mathbf{P}$ and intersecting $\conv \left( \bigcup \mathcal{P} \right)$ intersects a member of $\mathcal{P}$ in $\Ee^d$, i.e., there is no hyperplane parallel to some facet of $\mathbf{P}$ that strictly separates some elements of $\mathcal{P}$ from all the other elements of $\mathcal{P}$ in $\Ee^d$. Then, let $\Lambda(\mathcal{P}) > 0$ denote the smallest positive value $\lambda$ such that a translate of
$\lambda \left( \sum_{i=1}^n \tau_i \right) \mathbf{P}$ covers $\bigcup \mathcal{P}$.
\end{defn}

Clearly, every NS-family of finitely many positive homothetic copies of an arbitrary $d$-dimensional convex polytope is a WNS-family in $\Ee^d$, but not necessarily the other way around. Thus, the following theorem is a strengthening of Theorem~\ref{centrally symmetric convex bodies} for centrally symmetric convex polytopes. 

\begin{thm}\label{centrally symmetric convex bodies revisited}
For all $d\ge 2$ and $n\ge 2$, and for every $\mathbf{o}$-symmetric $d$-dimensional convex polytope $\mathbf{P}_0$ in $\Ee^d$ and every WNS-family $\mathcal{P} := \{ \xx_i + \tau_i \mathbf{P}_0\ |\  \xx_i\in \Ee^d, \tau_i>0, i=1,2,\ldots, n\}$ the inequality $\Lambda(\mathcal{P}) \leq 1$ holds. 
\end{thm}

\begin{rem}
The upper bound $1$ in Theorem~\ref{centrally symmetric convex bodies revisited} is a sharp one. Namely, equality can be attained for an $o$-symmetric $d$-dimensional convex polytope $\mathbf{P}_0$ if the centers of the non-overlapping homothetic copies of $\mathcal{P}$ lie on a line in $\Ee^d$ such that any two consecutive homothetic copies touch each other. 
\end{rem}

The following theorem is an extension of Theorem~\ref{simplex bound} to WNS-families of convex polytopes. In order to state it we need the following notation.

\begin{defn}\label{WNS-simplex}
For $d \geq 2$, let $\Lambda_{\rm simplex}^{(d)}$ denote the supremum of $\Lambda(\K)$, where $\K$ runs over the WNS-families of finitely many positive homothetic $d$-simplices in $\Ee^d$. 
\end{defn}

\begin{rem}
We note that $\Lambda_{\rm simplex}^{(d)}$ is a non-decreasing sequence of $d$. Indeed, let $\K$ be a WNS-family of finitely many $(d-1)$-simplices in a hyperplane of $\Ee^d$. Clearly, we can extend the elements of $\K$ to homothetic $d$-simplices such that each element is a facet of its extension. Then, denoting this extended family by ${\K}'$, we have $\Lambda (\K) = \Lambda({\K}')$, and ${\K}'$ is a WNS-family.
\end{rem}

\begin{thm}\label{WNS-simplex bound}
For all $d\ge 2$, $\Lambda_{\rm simplex}^{(d)}=\sup_{\mathcal{P}} \Lambda(\mathcal{P})$, where $\mathcal{P}$ runs over the WNS-families of finitely many positive homothetic copies of an arbitrary $d$-dimensional convex polytope $\mathbf{P}$ in $\Ee^d$.
\end{thm}

\begin{rem} \label{upper bound for simplices} Recall the following result of Akopyan, Balitskiy, and Grigorev \cite{ABG}. If $\mathcal{T} := \{ \xx_i + \tau_i \mathbf{T}\ |\  \xx_i\in \Ee^d, \tau_i>0, i=1,2,\ldots, n\}$ is an arbitrary WNS-family of the $d$-dimensional simplex $\mathbf{T}$ in $\Ee^d$, then $\Lambda(\mathcal{T})\leq (d+1)/2$ holds for all $d\geq 2$ and $n\geq 2$. Clearly,  this statement is equivalent to $\Lambda_{\rm simplex}^{(d)}\leq (d+1)/2$. 
\end{rem}

Notice that Remark~\ref{upper bound for simplices} together with Theorem~\ref{WNS-simplex bound} yield the following upper bound theorem. 

\begin{cor} \label {upper bound theorem} If $\mathcal{P} := \{ \xx_i + \tau_i \mathbf{P}\ |\  \xx_i\in \Ee^d, \tau_i>0, i=1,2,\ldots, n\}$ is an arbitrary WNS-family of the $d$-dimensional convex polytope $\mathbf{P}$ in $\Ee^d$, then $\Lambda(\mathcal{P})\leq (d+1)/2$ holds for all $d\geq 2 $ and $n\geq 2$.
\end{cor}

\subsection{A stability analogue of Corollary~\ref{upper bound theorem}}

In fact, Corollary~\ref{upper bound theorem} can be strengthened by establishing the following stability version. In order to state it, recall that for convex bodies $\mathbf{K}, \mathbf{L} \subset \Ee^d$, their Banach-Mazur distance $d_{BM}(\mathbf{K}, \mathbf{L})$ is defined by 
$$
d_{BM}(\mathbf{K}, \mathbf{L}):=\inf\{\lambda>0\ |\ \mathbf{K}+\mathbf{u}\subseteq T(\mathbf{L}+\mathbf{v})\subseteq \lambda(\mathbf{K}+\mathbf{u}\},
$$
where the infimum is taken over all invertible linear operators $T: \Ee^d\to \Ee^d$ and all vectors $\mathbf{u}, \mathbf{v}\in \Ee^d$.

\begin{thm}\label{thm:WNSstability}
Let $0 < \varepsilon < \frac{1}{16(1+d)}$. Let $\mathcal{P} := \{ \xx_i + \tau_i \mathbf{P}\ |\  \xx_i\in \Ee^d, \tau_i>0, i=1,2,\ldots, n\}$ be a WNS-family of positive homothetic copies of a $d$-dimensional convex polytope $\mathbf{P} \subset \Ee^d$. Assume that the smallest positive homothetic copy of $\mathbf{P}$ covering $\bigcup \mathcal{P}$ has homothety ratio greater than or equal to $\left( \frac{d+1}{2} - \varepsilon \right) \sum_{i=1}^n \tau_i$. Then $d_{BM}(\mathbf{P},\mathbf{S}) \leq 1+16(d+1) \varepsilon$ holds for all $d$-dimensional simplices $\mathbf{S}\subset \Ee^d$.
\end{thm}

\subsection{Maximizing the volume of translative WNS-families of cubes}\label{sec:linfty}

Next we investigate isoperimetric-type problems for translative WNS-families of the unit cube $\C_d := [0,1]^d$ in $\Ee^d$.
During this investigation, we call a family $\mathcal{F}$ of translates of $\C_d$ an \emph{integer family} if every vertex of every member of $\mathcal{F}$ is a point of the integer grid $\mathbb{Z}^d$.
Let $\emptyset\neq A\subset\Ee^d$ be a compact convex set, and $1\leq i\leq d$. We denote the {\it $i$-th} intrinsic volume of $A$ by $V_i(A)$. It is well known that $V_d(A)$ is the $d$-dimensional 
volume of $A$, $2V_{d-1}(A)$ is the surface area of $A$, and $\frac{2\omega_{d-1}}{d\omega_d}V_1(A)$ is equal to the mean width of $A$, where $\omega_d:=V_d(\mathbf{B}^d)=\frac{\pi^{\frac{d}{2}}}{\Gamma(1+\frac{d}{2})}$ with $\mathbf{B}^d:=\{\mathbf{x}\in\Ee^d \  | \|\mathbf{x}\|\leq 1\}$. For properties of intrinsic volumes, including Steiner's formula
$V_d\left(A+\epsilon\mathbf{B}^d\right)=\sum_{i=0}^{d}\omega_{d-i}V_i(A)\epsilon^{d-i}$, where $\epsilon >0$ and $V_0(A):=1$, we refer the interested reader to \cite{Sch14}. Our main result is the following. 

\begin{thm}\label{thm:WNScubes}
Let $n \geq 1$, and $1 \leq i \leq d$ be integers. Then the following holds.
\begin{itemize}
\item[(i)] There is a WNS-family $\mathcal{F}$ of $n$ translates of $\C_d$ that maximizes $V_i\left( \conv \left( \bigcup \mathcal{F}\right) \right)$ over the set of all WNS-families of $n$ translates of $\C_d$ such that the smallest axis-parallel box containing $\bigcup \mathcal{F}$ is $n \C_d$.
\item[(ii)] For every $n \geq 4$ and every WNS-family $\mathcal{F}$ of $n$ translates of $\C_2$, we have
\[
\area\left( \conv \left( \bigcup \mathcal{F}\right) \right)  \leq n^2 - 2n+4, \ \hbox{and} \  \perim\left( \conv \left( \bigcup \mathcal{F}\right) \right)  \leq 4+ 4\sqrt{n^2 - 4n+5},
\]
where $\area(\cdot)$ and $\perim(\cdot)$ refer to the area and perimeter of the given sets, respectively.
\end{itemize}
\end{thm}

\begin{rem}
We note that, as we will see in the proof of Theorem~\ref{thm:WNScubes}, the estimates in part (ii) of this theorem are sharp.
\end{rem}

\subsection{On coverings of weakly $k$-impassable families of convex polytopes}

\begin{defn}\label{Bezdek-Langi extended}
Let $\mathbf{P}$ be a $d$-dimensional convex polytope in $\Ee^d$ and let $\mathcal{P} := \{ \xx_i + \tau_i \mathbf{P}\ |\  \xx_i\in \Ee^d, \tau_i>0, i=1,2,\ldots, n\}$, where $d\ge 2$ and $n\ge 2$. Let $0\leq k\leq d-1$.
We call $\mathcal{P}$ a \emph{weakly $k$-impassable family}, in short, a \emph{$k$-WIP-family} if every $k$-dimensional affine subspace of $\Ee^d$ that is parallel to a facet of $\mathbf{P}$ (i.e., has a translate lying in the hyperplane of a facet of $\mathbf{P}$) and intersects $\conv \left( \bigcup \mathcal{P} \right)$, intersects also at least one member of $\mathcal{P}$. Then, let $\Lambda_k(\mathcal{P}) > 0$ denote the smallest positive value $\lambda$ such that a translate of
$\lambda \left( \sum_{i=1}^n \tau_i \right) \mathbf{P}$ covers $\bigcup \mathcal{P}$, where $\mathcal{P}$ is a $k$-WIP-family.
\end{defn}

Let $0\leq k\leq d-1$. Clearly, every $k$-IP-family of finitely many positive homothetic copies of an arbitrary d-dimensional
convex polytope is a $k$-WIP-family in $\Ee^d$, but not necessarily the other way around. Thus, the following theorem is a strengthening of Theorem~\ref{k-IP-family} for convex polytopes.

\begin{thm}\label{weak Bezdek-Langi}
Let $1\leq k\leq d-2$ and $n\geq 2$. If $\mathbf{P}$ is a $d$-dimensional convex polytope in $\Ee^d$ and $\mathcal{P} := \{ \xx_i + \tau_i \mathbf{P}\ |\  \xx_i\in \Ee^d, \tau_i>0, i=1,2,\ldots, n\}$ is a $k$-WIP-family, then  $\conv \left( \bigcup \mathcal{P} \right)$ slides freely in $\left( \sum_{i=1}^n \tau_i \right) \mathbf{P}$ (i.e., $\conv \left( \bigcup \mathcal{P} \right)$ is a summand of $\left( \sum_{i=1}^n \tau_i \right) \mathbf{P}$) and therefore $\Lambda_k(\mathcal{P}) \leq 1$, where equality holds for $n>1$ translates of a $d$-cube whose union is a rectangular box and whose centers lie on a line in $\Ee^d$.
\end{thm}

\subsection{Maximizing the tightness of NS-lattice arrangements}

L. Fejes T\'oth \cite{Fe76, Fe78} introduced the notion of closeness of a packing of balls in $\Ee^d$ and suggested the problem of finding the closest packing of equal balls in $\Ee^d$. For the status of this problem and several related questions we refer the interested reader to \cite{LiZo}. Here we recall the following variant of this concept, which leads us to an analogue of Theorem~\ref{FTL-Makai} as follows. Recall that for a family $\F$ of translates of a convex body $\KK$, the \emph{tightness} of $\F$ is defined as the largest value $\lambda > 0$ such that there is a homothetic copy $\xx + \lambda \KK$ that does not overlap $\bigcup \F$. Correcting the proof of Theorem 2.13 in \cite{KaLo}, Averkov and Wagner \cite[Theorem 2.5]{AvWa} proved the following theorem. We present a different and shorter proof.

\begin{thm}\label{max-tightness}
Let $\KK$ be an $\mathbf{o}$-symmetric plane convex body and let $L\in\mathcal{L}^2$. If $\F := \{ \xx+ \KK \  | \ \xx \in L \}$ is an NS-lattice arrangement, then the tightness of $\F$ is at most one. This is attained if and only if $\KK$ is a parallelogram with $L$ generating a chessboard-like lattice packing of $\KK$ having tightness one. 
\end{thm}

Let $L\in\mathcal{L}^d$, and let $\KK$ be an $\mathbf{o}$-symmetric convex body in $\Ed$. Let $\F := \{ \xx+ \KK \  | \ \xx \in L \}$. Then we denote by $T_{ns}(d)$ the supremum of the tightness of all NS-lattice arrangements $\F$ of all convex bodies $\KK$. 

\begin{thm}\label{tightness estimates}
For $d\geq 2$, we have
\[
d-1 \leq T_{ns}(d) \leq C d(1+\log d)
\]
for some universal constant $C > 0$.
\end{thm}

In the rest of the paper we prove the theorems stated in Section~\ref{new results}. Furthermore, in connection with NS-lattice arrangements, in the Appendix  we introduce the concept of weakly $k$-impassable lattice arrangements and make some basic observations on them.

\section{Proof of Theorem~\ref{NS-stability}}

Consider an NS-family $\mathcal{F}:= \{ \KK_i := \p_i + \tau_i \KK \ |\  \p_i\in \Ee^d, \tau_i>0, i=1,2,\ldots, n\}$. Assume that if a translate of $\bar{\tau} \KK$ covers $\bigcup \mathcal{F}$, then $\bar{\tau} \geq \left(\sum_{i=1}^n \tau_i\right) - \varepsilon$.

It is known (see \cite{BezLan16}) that the homothetic copy $\p + \tau \KK$, with $\tau := \sum_{i=1}^n \tau_i $ and center $\p:= \frac{\sum_{i=1}^n \tau_i \p_i}{ \sum_{i=1}^n \tau_i}$, contains $\bigcup \mathcal{F}$. Hence, according to our assumption, at least one point of $\bigcup \mathcal{F}$  is a boundary or an exterior point of $\p + (\tau-\varepsilon) \KK$. Let us choose a coordinate system in which $\o$ is a farthest point of $\bigcup \mathcal{F}$ from $\p$ with $\p$ sitting on the positive half of the first coordinate axis. Since the statement remains unchanged under any non-degenerated linear transformation of $\KK$, we may also assume that the projection of $\KK$ onto the first coordinate axis is the interval $[-1,1]$ and it coincides with the intersection of $\KK$ with the axis, and, in addition, that its projection onto the orthogonal complement of the axis does not contain any point farther than one from $\o$.

Observe that $\p=(x, \mathbf{0})$ for some $x > 0$ with $\mathbf{0}:=(0, \dots , 0)\in \Eu^{d-1}$ and let $(x_i,\mathbf{y}_i):=\p_i$ with $\mathbf{y}_i\in  \Eu^{d-1}$. Note that since $\mathcal{F}$ is an NS-family, its projection onto the first coordinate axis is a closed segment. Let us relabel the elements such that if $i \leq j$, then the left endpoint of the projection of $\p_i + \tau_i \KK$ is not greater than that of $\p_j + \tau_j \KK$.
Then for any $1\leq i\leq n$, we have $x_i \leq 2 \tau_1 + 2\tau_2 + \ldots + 2 \tau_{i-1} + \tau_i$ with $\tau_0:=0$.
This yields that
\[
x = \frac{\sum_{i=1}^n \tau_i x_i}{ \sum_{i=1}^n \tau_i} \leq \frac{\sum_{i=1}^n (2 \tau_1 + 2\tau_2 + \ldots + 2 \tau_{i-1} + \tau_i) \tau_i}{ \sum_{i=1}^n \tau_i}
= \frac{(\sum_{i=1}^n \tau_i )^2}{\sum_{i=1}^n \tau_i} = \sum_{i=1}^n \tau_i =\tau.
\]
On the other hand, our assumptions yield that $x \geq \left(\sum_{i=1}^n \tau_i\right) - \varepsilon=\tau - \varepsilon$. This implies that for every value of $i$, we have
\[
x_i \geq 2 \tau_1 + 2\tau_2 + \ldots + 2 \tau_{i-1} + \tau_i - \frac{\tau}{\tau_i} \varepsilon.
\]
Let $C'= \max \{ \frac{\tau}{\tau_i} \ | \ i=1,2,\ldots, n\}$. Then the above inequality yields that $\tau_{i-1}+\tau_{i}- C'\varepsilon \leq x_i-x_{i-1} \leq \tau_{i-1}+\tau_i+C'\varepsilon$. By our assumptions, $||\mathbf{y}_i|| \leq \tau$ for all values of $i$. This implies in a straightforward way that there exists $D:=D(\tau_1,\tau_2,\ldots,\tau_n, \KK)>0$ such that if $0<\varepsilon \leq D$ and $\p_{i-1} + \tau_{i-1} \KK$ and $\p_{i} + \tau_{i} \KK$ are disjoint, then there exists a hyperplane that strictly separates them and is disjoint from any other element of $\mathcal{F}$. Thus, the non-separability of $\mathcal{F}$ yields that $\p_{i-1} + \tau_{i-1} \KK$ and $\p_{i} + \tau_{i} \KK$ intersect for all values of $i$ whenever $0<\varepsilon \leq D$. On the other hand, since $\KK$ has $C^2$-class boundary with strictly positive curvature, there is a value $0 < \kappa$ such that the Gaussian curvature of $\KK$ at any point is at least $\kappa$.

According to our conditions, the point of $\KK_i$ with the smallest first coordinate is $(x_i-\tau_i, \mathbf{y}_i)$ and the point of $\KK_{i-1}$ with the largest first coordinate is $(x_{i-1}+\tau_{i-1},\mathbf{y}_{i-1})$. By Theorem 3.2.12. of \cite{Sch14}, the ball of radius $\frac{1}{\kappa}$ and center $(x_{i-1}+\tau_{i-1}-\frac{1}{\kappa},\mathbf{y}_{i-1})$ contains $\KK_{i-1}$. Similarly, the ball of radius $\frac{1}{\kappa}$ and center $(x_i-\tau_i+\frac{1}{\kappa},\mathbf{y}_i)$ contains $\KK_i$. Thus, the fact that $\KK_{i-1} \cap \KK_{i+1} \neq \emptyset$ yields that
\[
\left( (x_i - \tau_i) - (x_{i-1} + \tau_{i-1}) + \frac{2}{\kappa} \right)^2 + \|\mathbf{y}_i-\mathbf{y}_{i-1}\|^2 \leq \frac{4}{\kappa^2}.
\]
Since $- C' \varepsilon \leq (x_i - \tau_i) - (x_{i-1} + \tau_{i-1}) \leq C' \varepsilon$, from the above inequality we have for $0<\varepsilon \leq D$ that
\[
\left( \frac{2}{\kappa} - C' \varepsilon \right)^2 + \| y_i-y_{i-1}\|^2 \leq \frac{4}{\kappa^2},
\]
which yields that $\| y_i-y_{i-1} \| \leq \left(C'\left(\frac{4}{\kappa}+C'D\right)\right)^{\frac{1}{2}} {\varepsilon}^{\frac{1}{2}}$. Thus, the assertion follows by choosing the first coordinate axis as $L$ and setting $C:=(n-1)\left(C'\left(\frac{4}{\kappa}+C'D\right)\right)^{\frac{1}{2}} $.

Now we prove the second statement. Let us assume that the diameter of $\KK$ is two. Let us assume that the points $\p_2,\p_3, \ldots, \p_n$ are collinear and their line is parallel to a diameter of $\KK$, and their distances are $\tau_2+\tau_3, \ldots, \tau_{n-1}+\tau_n$. Let us assume that $\p_1+ \tau \KK$ and $\p_2 + \tau_2 \KK$ touch, and the angle $\alpha$ of the halflines through $\p_1$ and $\p_3$, starting at $\p_2$ is almost $\pi$. Let $\delta > 0$ denote the distance of $\p_2$ from the line passing through the points $\p_1$ and $\p_n$. If $\delta$ is sufficiently small, then the $\KK$-circumradius $\bar{\tau}$ of the NS-family $\mathcal{F}= \{ \p_i + \tau_i \KK : i=1,2,\ldots,n\}$ defined in this way is equal to $\frac{1}{2} (\tau_1 + \tau_n + \|\p_n-\p_1\|_{\KK}) \geq \frac{1}{2} (\tau_1 + \tau_n + \|\p_n-\p_1\|)$, where $\| \cdot \|_{\KK}$ denotes the norm generated by $\KK$, which is defined by $\|\xx\|_{\KK}:=\min\{\lambda \geq 0\  |\  \xx\in \lambda \KK\}$ for $\xx\in \Ee^d$. Note that
\[
\|\p_n-\p_1\| = \sqrt{\|\p_n-\p_2\|^2 - \delta^2} + \sqrt{\|\p_1-\p_2\|^2 - \delta^2} = \|\p_n - \p_2\| + \|\p_1 - \p_2\| - C_1\delta^2 + O(\delta^4)
\]
for some $C_1 > 0$. Let $\beta$ and $\gamma$ denote the angles of the triangle $\conv \{ \p_n \p_2, \p_1\}$ at $\p_n$ and $\p_1$, respectively.
Then $\beta = \arcsin \frac{\delta}{\|\p_n-\p_2\|} =: \Theta_1(\delta)$, and $\gamma = \arcsin \frac{\delta}{\|\p_2-\p_1\|} =: \Theta_2(\delta)$ implying that $\pi-\alpha = \Theta_1(\delta)+\Theta_2(\delta)$. The differentiability properties of $\bd(\KK)$ imply that $\|\p_1 - \p_2\| \geq  \tau_1 + \tau_2 - C_2 (\pi-\alpha)^2 = \tau_1 + \tau_2 - C_3 \delta^2$ for some positive constants $C_2, C_3$.
On the other hand, $\|\p_n - \p_2\|=\|\p_n - \p_2\|_{\KK}=\tau_n + 2\tau_{n-1} + \ldots + 2\tau_3 + \tau_2$, implying that
\[
\|\p_n-\p_1\| \geq (2\tau-\tau_n-\tau_1) - 2C \delta^2
\]
for some $C > 0$. This yields the second statement and completes the proof of Theorem~\ref{NS-stability}.

\section{Proof of Theorem~\ref{centrally symmetric convex bodies revisited}}

We follow the proof of Theorem~\ref{centrally symmetric convex bodies} in \cite{BezLan16}. So, we start with the following statement from \cite{GG45}.

\begin{lem}\label{lem:GG}
Let $\F := \left\{ [x_i-\tau_i, x_i + \tau_i ] | \tau_i > 0, i=1,2,\ldots, n\right\}$ be a family of closed intervals in $\Re$ such that $\bigcup \F$ is a single closed interval in $\Re$. Let $x := \left( \sum_{i=1}^n \tau_i x_i \right) / \left( \sum_{i=1}^n \tau_i \right)$. Then the closed interval
$ 
\left[ x - \sum_{i=1}^n \tau_i , x + \sum_{i=1}^n \tau_i \right]
$
covers $\bigcup \F$.
\end{lem}
Next, let $\xx := \left( \sum_{i=1}^n \tau_i \xx_i \right) / \left( \sum_{i=1}^n \tau_i \right)$, and set $\mathbf{P}' := \xx + \left( \sum_{i=1}^n \tau_i \right) \mathbf{P}_0$. We prove that $\mathbf{P}'$ covers $\bigcup \mathcal{P}$. For any line $L$ through the origin $\mathbf{o}$, let $\proj_L : \Ee^d \to L$ denote the orthogonal projection onto $L$, and let $h_{\mathcal{P} }: \Sph^{d-1} \to \Re$ and $h_{\mathbf{P}'} : \Sph^{d-1} \to \Re$ denote the support functions of $\conv\left( \bigcup \mathcal{P} \right)$ and $\mathbf{P}'$, respectively (for the definition of support function, see \cite{Sch14}).
If $L$ is orthogonal to any facet of $\mathbf{P}_0$, then $\proj_L \left( \bigcup \mathcal{P} \right)$ is a single interval, which, by Lemma~\ref{lem:GG}, is covered by $\proj_L \left( \mathbf{P}' \right)$.
Thus, for any $\mathbf{u} \in \Sph^{d-1}$ which is orthogonal to some facet of $\mathbf{P}_0$, we have that $h_{\mathcal{P}}(\mathbf{u}) \leq h_{\mathbf{P}'}(\mathbf{u})$. Since every $d$-dimensional convex polytope is the intersection of its closed supporting halfspaces bounded by the hyperplanes of its facets in $\Ee^d$, it readily follows that every $\xx_i + \tau_i \mathbf{P}_0$ is covered by $ \mathbf{P}'$ and therefore $\bigcup \mathcal{P} \subseteq \mathbf{P}'$, finishing the proof of Theorem~\ref{centrally symmetric convex bodies revisited}.

\section{Proof of Theorem~\ref{WNS-simplex bound}}

We shall need the following concept.

\begin{defn}
We call a $d$-dimensional convex polytope a {\rm $d$-dimensional generic convex polytope in $\Ee^d$ } if the outer unit normal vectors of any $d$ of its facets are linearly independent in $\Ee^d$, $d\geq 2$.
\end{defn}

\begin{lem}\label{WNS-sup}
If $\mathcal{P}$ (resp., $\mathcal{P}'$) runs over all WNS-families of finitely many positive homothetic copies of an arbitrary $d$-dimensional convex polytope $\mathbf{P}$ (resp., of an arbitrary $d$-dimensional generic convex polytope $\mathbf{P}'$) in $\Ee^d$, then $\sup_{\mathcal{P}} \Lambda(\mathcal{P})=\sup_{\mathcal{P}'} \Lambda(\mathcal{P}')$.
\end{lem}
\begin{proof} Clearly, $\sup_{\mathcal{P}} \Lambda(\mathcal{P})\geq\sup_{\mathcal{P}'} \Lambda(\mathcal{P}')$. So, we are left to show that $\sup_{\mathcal{P}} \Lambda(\mathcal{P})\leq\sup_{\mathcal{P}'} \Lambda(\mathcal{P}')$. We do it as follows. Let $\mathcal{P} := \{ \xx_i + \tau_i \mathbf{P}\ |\  \xx_i\in \Ee^d, \tau_i>0, i=1,2,\ldots, n\}$ be an arbitrary WNS-family of an arbitrary $d$-dimensional convex polytope $\mathbf{P}$ in $\Ee^d$. Without loss of generality we may assume that $\tau_1={\rm min}\{\tau_i\ |\ 1\leq i\leq n\}$ and $\xx_1=\mathbf{o}\in{\rm int} (\tau_1\mathbf{P})$. Let $\mathbf{P}_1:=\tau_1\mathbf{P}$. It will be convenient to use the notation $X_{\delta}:=\cup\{\mathbf{B}^d[\xx, \delta]\ |\ \xx\in X\}$ for any set $X\subseteq \Ee^d$, where $\mathbf{B}^d[\xx, \delta]:=\{\mathbf{y}\in\Ee^d \  | \|\mathbf{y}-\xx\|\leq \delta\}$. Furthermore, for any $0<\delta'<\frac{\pi}{2}$, let $f(\delta')\geq 0$ be the smallest nonnegative real with the property that for any two hyperplanes $H$ and $H'$ of $ \Ee^d$ having unit normal vectors with angle at most $\delta'$ and passing through an arbitrary common point of ${\rm conv}\left(\cup\mathcal{P}\right)$ the following containment holds: $H\cap{\rm conv}\left(\cup\mathcal{P}\right)\subset H'_{f(\delta')}$. We note that $\lim_{\delta'\to 0^+}f(\delta')=0$. The driving force of our proof is the following observation whose easy proof we leave to the reader. (See Figure~\ref{fig:WNS approximation}.)

\begin{figure}[ht]
\begin{center}
\includegraphics[width=0.9\textwidth]{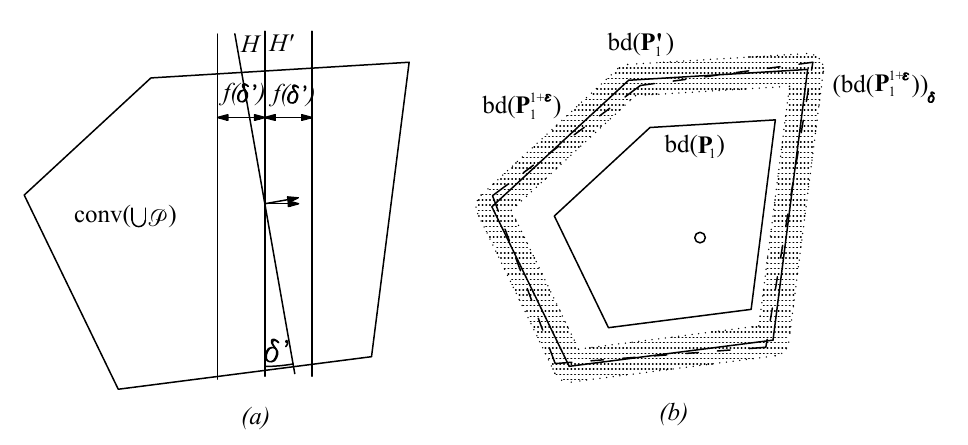}
\caption{An illustration for Proposition~\ref{approximation}. Panel (a): An illustration for the definition of $f(\delta')$. Panel (b): The notation in Proposition~\ref{approximation}. The sets $\bd(\mathbf{P}_1)$ and $\bd(\mathbf{P}_1^{1+\varepsilon})$ are drawn with solid lines. The dotted region indicates the set $\left({\rm bd}\ \mathbf{P}_1^{1+\epsilon}\right)_{\delta}$. The set $\bd(\mathbf{P}_1')$ is drawn with a dashed line.}
\label{fig:WNS approximation}
\end{center}
\end{figure}

\begin{prop}\label{approximation}
For every $\epsilon>0$ consider the convex polytope $\mathbf{P}_1^{1+\epsilon}:=(1+\epsilon)\mathbf{P}_1$. Then there exist a $d$-dimensional generic convex polytope $\mathbf{P}_1'$ and $\delta$ and $\delta'$ with $0<\delta'<\delta<\epsilon$ such that the following holds.
\begin{itemize}
\item[(a)]
There is a one-to-one correspondence between the facets of $\mathbf{P}_1'$ and $\mathbf{P}_1$ with the property that if the facet $F'$ of $\mathbf{P}_1'$ with outer unit normal vector $\mathbf{u}'$ corresponds to the facet $F$ of $\mathbf{P}_1$ with outer unit normal vector $\mathbf{u}$, then $\angle (\mathbf{u}, \mathbf{u}')\leq\delta'$, where $\angle (\mathbf{u}, \mathbf{u}')$ stands for the angle between $\mathbf{u}$ and $\mathbf{u}'$.
\item[(b)]
${\rm bd}\ \mathbf{P}_1'\subset\left({\rm bd}\ \mathbf{P}_1^{1+\epsilon}\right)_{\delta}.$
\item[(c)] $\left({\rm bd}\ \mathbf{P}_1^{1+\epsilon}\right)_{\delta+f(\delta')}\cap\mathbf{P}_1=\emptyset.$
\end{itemize} 
\end{prop}

Finally, for $i=2,\dots , n$ let $\mathbf{P}_i':=\xx_i+\frac{\tau_i}{\tau_1}\mathbf{P}_1'$ be a larger positive homothetic copy of the generic convex polytope $\mathbf{P}_1'$. Based on Proposition~\ref{approximation} it is easy to check that since $\mathcal{P}$ is a WNS-family therefore $\mathcal{P}_{\epsilon}:=\{ \mathbf{P}_i'\ |\  i=1,2,\ldots, n\}$ is a WNS-family as well. Moreover, by introducing the generic convex polytope $\mathbf{P}_{\epsilon}:=\frac{1}{\tau_1}\mathbf{P}_1'$ we see that $\mathcal{P}_{\epsilon}:=\{ \xx_i + \tau_i \mathbf{P}_{\epsilon}\ |\  \xx_i\in \Ee^d, \tau_i>0, i=1,2,\ldots, n\}$. Finally, the observation that $\mathbf{P}_{\epsilon}$ (resp., $\mathcal{P}_{\epsilon}$) converges to $\mathbf{P}$ (resp., $\mathcal{P}$) as $\epsilon\to 0^+$ implies in a straightforward way that $\lim_{\epsilon\to 0^+}\Lambda(\mathcal{P}_{\epsilon})=\Lambda(\mathcal{P})$. This completes the proof of Lemma~\ref{WNS-sup}. \end{proof}

Thus, we are left to prove that  $\Lambda_{\rm simplex}^{(d)}=\sup_{\mathcal{P}'} \Lambda(\mathcal{P}')$ holds, where $\mathcal{P}'$ runs over all WNS-families of finitely many positive homothetic copies of an arbitrary $d$-dimensional generic convex polytope $\mathbf{P}'$ in $\Ee^d$. As a first step, we show the following stronger version of Lutwak's containment theorem \cite{Lut}.

\begin{lem}\label{Lutwak-type}
Let $\mathbf{P}$ be a $d$-dimensional generic convex polytope in $ \Ee^d$ and let $\mathbf{K}$ be a convex body in $ \Ee^d$. Assume that every $d$-dimensional simplex which is the intersection of the supporting halfspaces of some $d+1$ facets of $\mathbf{P}$ contains a translate of $\mathbf{K}$. Then $\mathbf{P}$ contains a translate of $\mathbf{K}$.
\end{lem} 

\begin{proof} We follow the method of \cite{Lut}. Let the outer unit normal vectors of the facets of $\mathbf{P}$ be $\mathbf{u}_1, \mathbf{u}_2, \dots , \mathbf{u}_s$. Moreover, for each $i$ with $1\leq i\leq s$, let $H_i^+:=\{\mathbf{x}\in \Ee^d | \langle \mathbf{x}, \mathbf{u}_i\rangle\leq h_{\mathbf{P}}(\mathbf{u}_i)\}$ be the closed supporting halfspace of $\mathbf{P}$ bounded by the hyperplane of the facet with outer unit normal vector $\mathbf{u}_i$, where $\langle\cdot ,\cdot\rangle$ stands for the standard inner product of  $\Ee^d$ and $h_{\mathbf{P}}(\cdot )$ denotes the support function of $\mathbf{P}$. Furthermore, let $$\mathbf{K}_i:=\{\mathbf{x}\in \Ee^d | \mathbf{x}+\mathbf{K}\subset H_i^+\}.$$ Clearly, each $\mathbf{K}_i$ is a closed halfspace in $ \Ee^d$. Next, let $\mathbf{K}_{i_1}, \mathbf{K}_{i_2},\dots , \mathbf{K}_{i_{d+1}}$ be any $d+1$ members of the family $\{\mathbf{K}_1, \mathbf{K}_2,\dots , \mathbf{K}_s\}$. As any $d$ of the vectors $\mathbf{u}_{i_1}, \mathbf{u}_{i_2}, \dots , \mathbf{u}_{i_{d+1}}$ are linearly independent, there are two possibilities: either that $\bigcap_{j=1}^{d+1}H_{i_j}^+$ contains a translate of a Euclidean ball of radius $r$ for all $r>0$, or that $\bigcap_{j=1}^{d+1}H_{i_j}^+$ is a simplex containing $\mathbf{P}$. In the first case there exists $\mathbf{x}_0\in  \Ee^d$ such that $\mathbf{x}_0+\mathbf{K}\subseteq \bigcap_{j=1}^{d+1}H_{i_j}^+$. In the second case, the hypothesis of Lemma~\ref{Lutwak-type} guarantees an $\mathbf{x}_0\in  \Ee^d$ such that $\mathbf{x}_0+\mathbf{K}\subseteq \bigcap_{j=1}^{d+1}H_{i_j}^+$. Thus, in either case $\mathbf{x}_0\in \bigcap_{j=1}^{d+1}\mathbf{K}_{i_{j}}$, and Helly's theorem (\cite{Sch14}) shows that there exists $\mathbf{x}\in \Ee^d$ such that $\mathbf{x}\in \bigcap_{i=1}^{s}\mathbf{K}_i$. Hence, $\mathbf{x}+\mathbf{K}\subseteq\bigcap_{i=1}^{s}H_i^+=\mathbf{P}$, finishing the proof of Lemma~\ref{Lutwak-type}.
\end{proof}

Now, we proceed with the proof of Theorem~\ref{WNS-simplex bound} as follows. Let $\mathcal{P} := \{ \xx_i + \tau_i \mathbf{P}\ |\  \xx_i\in \Ee^d, \tau_i>0, i=1,2,\ldots, n\}$ be an arbitrary WNS-family of the $d$-dimensional generic convex polytope $\mathbf{P}$ in $\Ee^d$. Moreover, let $\Delta(\mathbf{P})$ be any circumscribed $d$-dimensional simplex of $\mathbf{P}$, that is, let $\Delta(\mathbf{P})$ be a $d$-dimensional simplex which the intersection of the supporting halfspaces of some $d+1$ facets of $\mathbf{P}$. Then $\Lambda_{\rm simplex}^{(d)} ( \sum_{i=1}^n \tau_i ) \Delta(\mathbf{P})$ is a circumscribed simplex of $\Lambda_{\rm simplex}^{(d)} ( \sum_{i=1}^n \tau_i ) \mathbf{P}$ and $\xx_i + \tau_i \Delta(\mathbf{P})$ is a circumscribed simplex of $ \xx_i + \tau_i \mathbf{P}$ for all $i=1,2,\ldots, n$. Furthermore, $\{ \xx_i + \tau_i \Delta(\mathbf{P})\ |\  \xx_i\in \Ee^d, \tau_i>0, i=1,2,\ldots, n\}$ is a WNS-family of the $d$-dimensional simplex $\Delta(\mathbf{P})$ in $\Ee^d$. Hence, $\Lambda_{\rm simplex}^{(d)} ( \sum_{i=1}^n \tau_i ) \Delta(\mathbf{P})$ has a translate that covers $\bigcup \{ \xx_i + \tau_i \Delta(\mathbf{P})\ |\  \xx_i\in \Ee^d, \tau_i>0, i=1,2,\ldots, n\}\supseteq   \bigcup \mathcal{P}$. Thus, $\Lambda_{\rm simplex}^{(d)}  ( \sum_{i=1}^n \tau_i ) \Delta(\mathbf{P})$ has a translate that covers the convex body $\conv\left( \bigcup \{ \xx_i + \tau_i \Delta(\mathbf{P})\ |\  \xx_i\in \Ee^d, \tau_i>0, i=1,2,\ldots, n\} \right)\supseteq   \conv\left(\bigcup \mathcal{P}\right)$. Since $\Lambda_{\rm simplex}^{(d)}  ( \sum_{i=1}^n \tau_i ) \mathbf{P}$ has the same outer unit normal vectors as $\mathbf{P}$ and $\Delta(\mathbf{P})$ was an arbitrary circumscribed $d$-dimensional simplex of $\mathbf{P}$, Lemma~\ref{Lutwak-type} implies in a straghforward way that $\Lambda_{\rm simplex}^{(d)}  ( \sum_{i=1}^n \tau_i ) \mathbf{P}$ has a translate that covers $ \conv\left(\bigcup \mathcal{P}\right)\supseteq\bigcup \mathcal{P}$. This shows that $\Lambda_{\rm simplex}^{(d)}\geq \sup_{\mathcal{P}} \Lambda(\mathcal{P})$, where $\mathcal{P}$ runs over the WNS-families of finitely many positive homothetic copies of an arbitrary $d$-dimensional generic convex polytope $\mathbf{P}$ in $\Ee^d$. As the inequality $\Lambda_{\rm simplex}^{(d)}\leq \sup_{\mathcal{P}} \Lambda(\mathcal{P})$ holds trivially, the proof of Theorem~\ref{WNS-simplex bound} is complete. 

\section{Proof of Theorem~\ref{thm:WNSstability}}

For the proof we use the following parameter of asymmetry.

\begin{defn}\label{defn:Minkowski}
Let $\mathbf{K} \subset \Ee^d$ be a convex body. The quantity
\[
\sigma(\mathbf{K}) := \min_{\mathbf{q} \in \inter (\mathbf{K}) } \min \left\{ \mu > 0 \ |\  \mathbf{K}-\mathbf{q} \subseteq - \mu (\mathbf{K}-\mathbf{q}) \right\}
\]
is called the \emph{Minkowski measure of asymmetry} of $\mathbf{K}$.
\end{defn}

It is well known (Lemma 2.2 in \cite{ABG}) that for any convex body $\mathbf{K} \subset \Ee^d$, $\sigma(\mathbf{K}) \leq d$. The following stronger form of this observation is proved by Guo in \cite[Theorem A]{Guo}.

\begin{thm}[Guo, 2005]\label{thm:Guo}
Let $0 < \varepsilon < \frac{1}{8(1+d)}$. Let $\mathbf{K} \subset \Ee^d$ be a convex body. If $\sigma(\mathbf{K}) \geq d-\varepsilon$, then $d_{BM}(\mathbf{K}, \mathbf{S}) \leq 1+8(d+1) \varepsilon$ holds for all $d$-dimensional simplices $\mathbf{S}\subset \Ee^d$.
\end{thm}
Clearly, Theorem~\ref{thm:WNSstability} follows from Theorem~\ref{thm:Guo} and the following strengthening of Theorem 2.1 from \cite{ABG}.

\begin{thm}\label{strongerABG}
Let $\mathbf{P} \subset \Ee^d$ be a convex polytope. Let $\mathcal{P} := \{ \xx_i + \tau_i \mathbf{P}\ |\  \xx_i\in \Ee^d, \tau_i>0, i=1,2,\ldots, n\}$ be a WNS-family of homothetic copies of $\mathbf{P}$. Then there is a translate of $\frac{\sigma(\mathbf{P})+1}{2} \left( \sum_{i=1}^n \tau_i\right) \mathbf{P}$ that covers $\bigcup \mathcal{P}$.
\end{thm}

\begin{proof}
We follow the proof of Theorem 2.1 in \cite{ABG}. First, we remark that it is an elementary exercise to show that for any convex body $\mathbf{K} \subset \Ee^d$ with $\mathbf{o}\in{\rm int}\ \mathbf{K}$,
\[
\sigma (\mathbf{K})= \min_{\mathbf{q} \in \inter (\mathbf{K}) } \min \left\{ \mu > 0\ |\  \left( \mathbf{K}-\mathbf{q} \right)^\circ \subseteq -\mu \left( \mathbf{K}-\mathbf{q} \right)^\circ \right\},
\]
where $\mathbf{Q}^\circ:=\{\mathbf{x}\in \Ee^d\ |\ \langle\mathbf{x},\mathbf{y}\rangle\leq 1 \ {\rm for \ all}\ \mathbf{y}\in\mathbf{Q}\}$ denotes the polar of the convex body $\mathbf{Q}$ with $\mathbf{o}\in{\rm int}\ \mathbf{Q}$. Based on this observation, by shifting the origin $\mathbf{o}$ we may assume that $\mathbf{P}^\circ \subseteq -\sigma \mathbf{P}^\circ$, where $\sigma := \sigma(\mathbf{P})$.

Let $\xx:= \frac{\sum_{i=1}^n \tau_i \xx_i}{ \sum_{i=1}^n \tau_i}$. Consider the homothetic copy $\xx +\frac{\sigma+1}{2} \left( \sum_{i=1}^n \tau_i\right) \mathbf{P}$. We show that it covers $\bigcup \mathcal{P}$. Suppose for contradiction that it does not. In this case there is a hyperplane $H$ strictly separating a point $\mathbf{p} \in \conv \left( \bigcup \mathcal{P} \right)$ from $\xx +\frac{\sigma+1}{2} \left( \sum_{i=1}^n \tau_i\right) \mathbf{P}$. Since $\mathbf{P}$ is a convex polytope, we may assume that $H$ is parallel to a facet of $\mathbf{P}$.

Let $\pi: \Ee^d \to L$ be the orthogonal projection onto a line $L$ perpendicular to $H$. Assume that the projection of $\mathbf{o}$ divides the projection $\pi(\mathbf{P})$ in the ratio $1:s$, where $s \geq 1$. As $\mathbf{P}^\circ \subseteq -\sigma \mathbf{P}^\circ$, we have $s \in [1,\sigma]$. 
In the following, we regard the line $L$ as the real line $\Re$. Let $[a_i,b_i] := \pi(\xx_i + \tau_i \mathbf{P})$, $c_i := \pi(\xx_i)$, $l_i := b_i-a_i$, and $L:=\sum_{i=1}^n l_i$. Observe that with this notation we have (using a suitable orientation of $L$) that the $l_i$ are proportional to the $\tau_i$, and $s(c_i-a_i)=b_i-c_i$. In addition, let $c:=\pi\left(\xx\right)=\pi\left(\frac{\sum_{i=1}^n l_i c_i}{ \sum_{i=1}^n l_i}\right)$ and $I=[a,b] := \pi \left( \xx +\frac{\sigma+1}{2} \left( \sum_{i=1}^n \tau_i\right) \mathbf{P}\right)$. Note that $I$ is of length $\frac{\sigma+1}{2}L$, and it is divided by $c$ in the ratio $1:s$.

Let $c_i':=\frac{a_i+b_i}{2}$. Since $\mathcal{P}$ is a WNS-family, the union of the segments $[a_i,b_i]$ is a closed interval in $\Re$. Thus the segment $I':=[a',b']$ of length $L$ and midpoint $c':=\frac{\sum_{i=1}^n l_i c_i'}{L}$ covers $\bigcup [a_i,b_i]=\pi\left(\mathcal{P}\right)$ (see \cite[Lemma]{GG45}, or \cite[Lemma 3]{BezLan16} or \cite[Lemma 1.2]{ABG}). Then it follows through a sequence of elementary inequalities that $I' \subseteq I$ (see the proof of \cite[Theorem 2.1]{ABG}), which is a contradiction since $\pi(\mathbf{p})\in I'$ and $\pi(\mathbf{p})\notin I$.
\end{proof}

\begin{rem}
As for any convex body $\mathbf{K} \subset \Ee^d$, we have the inequality $\sigma(\mathbf{K}) \leq d$, therefore Theorem~\ref{strongerABG} implies Corollary~\ref{upper bound theorem} in a straightforward way, thereby providing yet another proof for it.
\end{rem}

\section{Proof of Theorem~\ref{thm:WNScubes}}

First, we prove (i).
Consider a WNS-family $\mathcal{F}:= \{ \p_k + \C_d \ |\  k=1,2,\ldots, n \}$ of $n$ translates of $\C_d$, where $\p_k := (p_k^1, p_k^2, \ldots, p_k^d)\in\Ee^d$. Let $\e_1, \e_2, \ldots, \e_d$ denote the standard orthonormal basis vectors in $\Ee^d$.

\begin{figure}[ht]
\begin{center}
\includegraphics[width=0.4\textwidth]{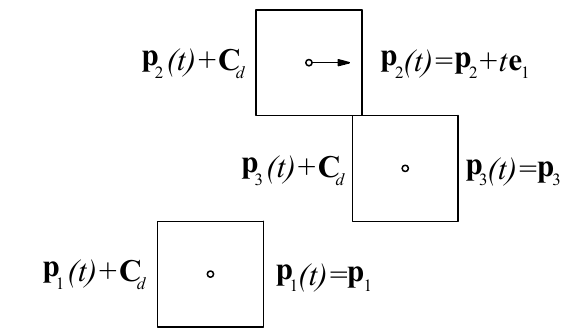}
\caption{The shadow system defined in the proof of Theorem~\ref{thm:WNScubes}}
\label{fig:shadowsystem}
\end{center}
\end{figure}

First, we show that there is an integer WNS-family $\mathcal{F}'$ satisfying $V_i\left( \conv \left( \bigcup \mathcal{F}\right) \right) \leq V_i\left( \conv \left( \bigcup \mathcal{F}'\right) \right)$. If $\mathcal{F}$ is a translate of an integer family, then the statement clearly holds. Assume that this condition is not satisfied. Then, without loss of generality, we may assume that there are translates $\p_j + \C_d$ and $\p_k+\C_d$ such that $p_k^1-p_j^1$ is not an integer. For any $0 \leq \tau < 1$, let $S_{\tau}$ denote the set of indices $j$ with the property that $p_j^1 \equiv \tau$ modulo $1$. Let $T:= \{ \tau_1, \tau_2, \ldots, \tau_m \}$ denote the set of values $\tau$ for which $S_\tau \neq \emptyset$. Without loss of generality, we may assume that $0 \leq \tau_1 < \tau_2 < \ldots < \tau_m$, and set $\tau_{m+1}:=\tau_1+1$. Due to our assumptions, $m \geq 2$. Consider some value $2 \leq k \leq m$. Following \cite{RS} and \cite{Shephard}, we define a shadow system $\mathcal{F}(t)$ on the interval $t \in [\tau_{i-1}-\tau_i, \tau_{i+1}-\tau_i]$ as follows:
\begin{itemize}
\item[(1)] If $j \notin S_{\tau_k}$, then $\p_j(t):=\p_j$.
\item[(2)] If $j \in S_{\tau_k}$, then $\p_j(t):= \p_j + t \e_1$.
\item[(3)] We set $\mathcal{F}(t) := \{ \p_j(t)+ \C_d \ |\  j=1,2,\ldots, n \}$ (see Figure~\ref{fig:shadowsystem}).
\end{itemize}
Then, for every value of $t$, $\mathcal{F}(t)$ is a WNS-family of $n$ translates of $\C_d$. Thus, $$V_i\left( \conv \left( \bigcup \mathcal{F}(t)\right) \right)$$ is a convex function of $t$ (see \cite{Shephard}). Since a convex function, defined on an interval, attains its maximum at one of the endpoints of the interval, we may replace $\mathcal{F}$ with another WNS-family in which the number of the classes of the indices of the elements is strictly less than $m$. Repeating the procedure we obtain a WNS-family for which the difference of the first coordinates of any two translation vectors is an integer. Applying the same consideration for all coordinates of the translation vectors, we obtain the desired statement, and from now on we assume that $\mathcal{F}$ is an integer WNS-family.

Next, we show that there is a WNS-family $\mathcal{F}'$ satisfying $$V_i\left( \conv \left( \bigcup \mathcal{F}\right) \right) \leq V_i\left( \conv {\left( \bigcup \mathcal{F}'\right)} \right)$$ such that the smallest axis-parallel box $\mathbf{B}$ containing $\bigcup \mathcal{F}'$ is a translate of $n \C_d$. Assume that $\mathcal{F}$ does not satisfy this property. Let $\mathbf{B} = \bigtimes_{j=1}^d [a_j,b_j]$ for some integers $a_j,b_j$. Without loss of generality, assume that $b_1 - a_1 < n$. Then there are some indices $j \neq k$ such that $p_j^1 = p_k^1$. Then, moving $\p_j + \C_d$ parallel to $\e_1$ at a constant speed such that the first coordinate $t$ of the translation vectors runs over the interval $[a_1-1,b_1]$, and keeping all other elements of $\mathcal{F}$ fixed, we define a shadow system $\mathcal{F}(t)$ whose every element is a WNS-family. By the previous argument, replacing $\mathcal{F}$ with $\mathcal{F}(a_1-1)$ or $\mathcal{F}(b_1)$, we obtain a WNS-family such that the smallest axis-parallel box containing it is $[a_1-1,b_1] \times \bigtimes_{j=2}^d [a_j,b_j]$ or $[a_1,b_1+1] \times \bigtimes_{j=2}^d [a_j,b_j]$.
Now, repeating this procedure until $b_j-a_j < n$ for some value of $j$ finishes the proof of (i).

Next, we prove (ii). First, note that by (i), it is sufficient to prove (ii) under the condition that the smallest axis-parallel rectangle containing $\mathcal{F}$ is $\mathbf{B}:=[0,n]^2$. Let us define a WNS-family $\mathcal{F}_n$ as follows. Let $\C' :=[1,n-1]^2$. Glue four unit squares to $\C'$ such that the bottom left corners of the squares are the points $(1,0)$, $(n-1,1)$, $(n-2,n-1)$ and $(0,n-2)$. Then, if we add $n-4$ axis-parallel unit squares along one of the two diameters of the square $[2,n-2]^2$, then we obtain a WNS-packing $\mathcal{F}_n$ of $n$ translates of $\C_2$. An elementary computation shows that
\[
\area\left( \conv \left( \bigcup \mathcal{F}_n\right) \right)  = n^2 - 2n+4, \quad \hbox{and} \quad \perim\left( \conv \left( \bigcup \mathcal{F}_n \right) \right)  = 4+ 4\sqrt{n^2 - 4n+5}.
\]

We prove the statement about area. Consider the case that no vertex of $\mathbf{B}$ belongs to an element of $\mathcal{F}$. Take the side of $\mathbf{B}$ on the line $x=0$. If we move along this line the element $\C(i)$ of $\mathcal{F}$ adjacent to it, we define a shadow system, and thus, it follows that the area $\mathcal{F}$ is maximal if the left bottom corner of $\C(i)$ is either $(0,1)$ or $(0,n-2)$. Applying a similar consideration to every side of $\mathbf{B}$, a simple computation yields the assertion. If a vertex of $\mathbf{B}$ belongs to an element of $\mathcal{F}$, we may apply a slightly modified variant of this consideration.
Finally, to prove the statement about perimeter, we follow the argument in the proof for area. This completes the proof of Theorem~\ref{thm:WNScubes}.

\section{Proof of Theorem~\ref{weak Bezdek-Langi}}

We start by recalling the following notations and concepts. Let the $(d-1)$-dimensional unit sphere centered at the origin of $\Ed$ be denoted by $\Sd:=\{\mathbf{u}\in\Ed | \|\mathbf{u}\|=1\}$. Let $\mathbf{K}$ be a convex body in $\Ed$ and $H^+$ be a closed halfspace of $\Ed$ bounded by the hyperplane $H$. We call $\mathbf{u}\in\Sd$ the outer unit normal vector of $H^+$ if $H^+$ is a translate of $\{\mathbf{x}\in\Ed | \langle \mathbf{x},\mathbf{u}\rangle\leq 0\}$. Then $H$ is a translate of $\{\mathbf{x}\in\Ed | \langle \mathbf{x},\mathbf{u}\rangle= 0\}$. Finally, we call $H^+$ (resp., $H$) the supporting halfspace (resp., supporting hyperplane) of $\mathbf{K}$ with outer unit normal vector $\mathbf{u}$, if $\mathbf{K}\subset H^+$ and $H\cap \mathbf{K}\neq\emptyset$. In this case, the set $F(\mathbf{K}, \mathbf{u}):=H\cap\mathbf{K}$ is called the support set or in short, (exposed) face of $\mathbf{K}$ with outer unit normal vector $\mathbf{u}$. For the proof that follows we need Theorem 3.2.11 from \cite{Sch14} which we state as follows.
\begin{lem}[Schneider, 2014]\label{Schneider Lemma}
Let $\mathbf{K}$ be a convex body and $\mathbf{Q}$ be a $d$-dimensional convex polytope in $\Ed$, $d>1$. Then $\mathbf{Q}$ is a summand of $\mathbf{K}$ (i.e., $\mathbf{Q}$ slides freely inside $\mathbf{K}$) if and only if the face $F(\mathbf{K}, \mathbf{u})$ contains a translate of $F(\mathbf{Q}, \mathbf{u})$ whenever $F(\mathbf{Q}, \mathbf{u})$ is an edge of $\mathbf{Q}$ for $\mathbf{u}\in\Sd$.
\end{lem}

Now, we turn to the proof of Theorem~\ref{weak Bezdek-Langi}. Clearly, if $\mathcal{P} = \{ \xx_i + \tau_i \mathbf{P}\ |\  \xx_i\in \Ee^d, \tau_i>0, i=1,2,\ldots, n\}$ is a $k$-WIP-family for $1\leq k\leq d-2$, then $\mathcal{P}$ is a $(d-2)$-WIP-family as well. Thus, it is sufficient to prove Theorem~\ref{weak Bezdek-Langi} for any $\mathcal{P}=\{ \xx_i + \tau_i \mathbf{P}\ |\  \xx_i\in \Ee^d, \tau_i>0, i=1,2,\ldots, n\}$ which is a $(d-2)$-WIP-family, where $n\geq 2$. In order to prove it we need the following statement.

\begin{lem}\label{covering an edge by edges}
If $\mathbf{P}$ is a $d$-dimensional convex polytope in $\Ee^d$ and $\mathcal{P} := \{ \xx_i + \tau_i \mathbf{P}\ |\  \xx_i\in \Ee^d, \tau_i>0, i=1,2,\ldots, n\}$ is a $(d-2)$-WIP-family for $d\geq 3$ and $n\geq 2$, then every edge of $\conv  \left( \bigcup \mathcal{P} \right)$ is covered by $\bigcup \mathcal{P}$.
\end{lem}

\begin{proof} We prove it by contradiction. Assume that the $d$-dimensional convex polytope $\mathbf{Q}:={\rm conv}\left(\bigcup  \mathcal{P}\right)$ has an edge say, the closed line segment $[\mathbf{x},\mathbf{y}]$ connecting the vertices $\mathbf{x},\mathbf{y}\in{\rm vert}\left(\mathbf{Q}\right)$ such that $[\mathbf{x},\mathbf{y}]\setminus\left(\bigcup  \mathcal{P}\right)\neq\emptyset$, where ${\rm vert}(\cdot)$ refers to the set of vertices of the corresponding convex polytope. This means that there are vertices $\mathbf{x'}\in{\rm vert}\left( \xx_i + \tau_i \mathbf{P}\right)$ and $\mathbf{y'}\in{\rm vert}\left( \xx_j + \tau_j \mathbf{P}\right)$ for some $1\leq i,j\leq n$ such that $\emptyset\neq (\mathbf{x'},\mathbf{y'})\subsetneq [\mathbf{x},\mathbf{y}]$ and $(\mathbf{x'},\mathbf{y'})\cap\left(\bigcup  \mathcal{P}\right)=\emptyset$, where $(\mathbf{x'},\mathbf{y'})$ denotes the open line segment with endpoints $\mathbf{x'}$ and $\mathbf{y'}$. Without loss of generality we may assume that as we move from $\mathbf{x}$ towards $\mathbf{y}$ first we visit $\mathbf{x'}$ and then $\mathbf{y'}$. Furthermore, let $H^+$ (resp., $H$) be a supporting halfspace (resp., supporting hyperplane) of $\mathbf{Q}$ with outer unit normal vector $\mathbf{u}\in\Sd$ such that $F(\mathbf{Q}, \mathbf{u})=H\cap\mathbf{K}=[\mathbf{x},\mathbf{y}]$ (see Figure~\ref{fig:covering edges}).

\begin{figure}[ht]
\begin{center}
\includegraphics[width=0.5\textwidth]{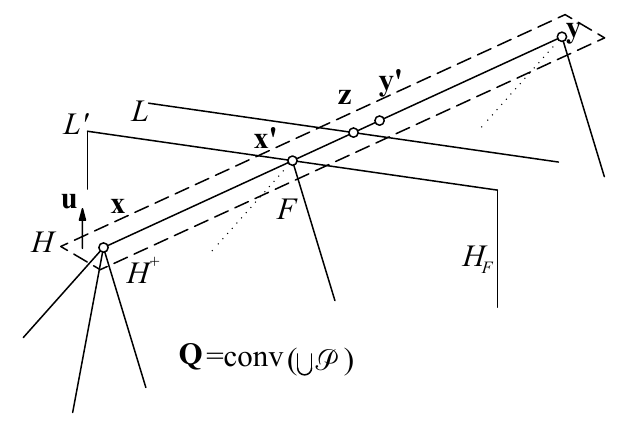}
\caption{Every edge of $\mathbf{Q}=\conv \left( \bigcup \mathcal{P} \right)$ is covered by $\bigcup \mathcal{P}$.}
\label{fig:covering edges}
\end{center}
\end{figure}

Now, let $\mathbf{T}_{\xx_i + \tau_i \mathbf{P}}(\mathbf{x'})$ be the tangent cone of $\xx_i + \tau_i \mathbf{P}$ at the vertex $\mathbf{x'}$, which is the union of all closed halflines starting at $\mathbf{x'}$ and passing through the points of $\xx_i + \tau_i \mathbf{P}$. Clearly, $\mathbf{T}_{\xx_i + \tau_i \mathbf{P}}(\mathbf{x'})$ is a $d$-dimensional convex polyhedral cone with apex $\mathbf{x'}$. By assumption, there is a facet $F$ of $\mathbf{T}_{\xx_i + \tau_i \mathbf{P}}(\mathbf{x'})$ whose hyperplane $H_F$ separates $(\mathbf{x'},\mathbf{y'})$ from $\mathbf{T}_{\xx_i + \tau_i \mathbf{P}}(\mathbf{x'})$ with $\mathbf{x'}\in F\subsetneq H_F$. It follows that $L':=H\cap H_F$ is a $(d-2)$-dimensional affine subspace in $\Ed$ with the property that $L'\cap\mathbf{Q}=\{\mathbf{x'}\}$. Finally, let $\mathbf{z}\in{\rm relint(\mathbf{x'},\mathbf{y'})}$ and let $L$ be the translate of $L'$ that passes through $\mathbf{z}$. Clearly, $L\cap\mathbf{Q}=\{\mathbf{z}\}$ and therefore $L\cap \left(\bigcup  \mathcal{P}\right)=\emptyset$. As by assumption $\mathcal{P}$ is a $(d-2)$-WIP-family therefore we have arrived at a contradiction. This completes the proof of Lemma~\ref{covering an edge by edges}.
\end{proof}
Finally, we finish the proof of Theorem~\ref{weak Bezdek-Langi} as follows.  Let $H^+$ (resp., $H$) be a supporting halfspace (resp., supporting hyperplane) of $\mathbf{Q}$ with outer unit normal vector $\mathbf{u}\in\Sd$ such that $F(\mathbf{Q}, \mathbf{u})=H\cap\mathbf{K}=[\mathbf{x},\mathbf{y}]$ is an edge of $\mathbf{Q}={\rm conv}\left(\bigcup  \mathcal{P}\right)$. As by assumption $\mathcal{P}$ is a $(d-2)$-WIP-family therefore Lemma~\ref{covering an edge by edges} implies in a straightforward way that the edge $[\mathbf{x},\mathbf{y}]$ can be translated into the parallel edge of $\left( \sum_{i=1}^n \tau_i \right) \mathbf{P}$ with supporting halfspace $H^+$ (resp., supporting hyperplane $H$) of outer unit normal vector $\mathbf{u}\in\Sd$. Thus, Lemma~\ref{Schneider Lemma} yields that $\mathbf{Q}$ slides freely in $\left( \sum_{i=1}^n \tau_i \right) \mathbf{P}$ (i.e., $\mathbf{Q}$ is a summand of $\left( \sum_{i=1}^n \tau_i \right) \mathbf{P}$), finishing the proof of 
Theorem~\ref{weak Bezdek-Langi}.

\section{Proof of Theorem~\ref{max-tightness}}

By Theorem 1 of \cite{Ma78}, the density $\delta(\F)$ of $\F$ satisfies
\begin{equation}\label{eq:Makai}
\delta(\F) \geq \frac{\vol(\KK)\vol(\KK^{\circ})}{16 \delta(\KK^{\circ})},
\end{equation}
where $\delta(\KK^{\circ})$ is the density of a densest lattice packing of the polar body $\KK^{\circ}$ of $\KK$. (Recall that $\mathbf{K}^\circ:=\{\mathbf{x}\in \Ee^d\ |\ \langle\mathbf{x},\mathbf{y}\rangle\leq 1 \ {\rm for \ all}\ \mathbf{y}\in \mathbf{K}  \}$.) Thus, using the fact that $\delta(\KK^{\circ}) \leq 1$ and by Mahler's result (\cite{GoSc}) according to which $\vol(\KK) \vol(\KK^{\circ}) \geq 8$, we obtain that
\[
\delta(\F) \geq \frac{1}{2}.
\]

Now, assume that the tightness of $\F$ is at least one. Then there is a translate $\xx + \KK$ that does not overlap $\Lambda + \KK$. Thus, $L + \xx + \KK$ does not overlap $L + \KK$. Since these sets are translates of each other, it follows that $\delta(\F) \leq \frac{1}{2}$. Combining it with the previous inequality for $\delta(\F)$, we have that $\delta(\F) = \frac{1}{2}$. Thus, the results about the $2$-dimensional Mahler Conjecture (\cite{GoSc}) yield that $\KK$ is a parallelogram, and that $(L + \KK) \cup (L + \xx + \KK) = \Ee^2$. From this an elementary consideration yields that $L$ is generated by two linearly independent vertices of $\KK$, and also that the tightness of $\F$ is one. This completes the proof of Therem~\ref{max-tightness}.

\section{Proof of Theorem~\ref{tightness estimates}}

Let $L\in\mathcal{L}^d$ and let $L^*:=\{\mathbf{x}\in\Ee^d\ |\ \langle\mathbf{x},\mathbf{y}\rangle\in\mathbb{Z}\ {\rm for\ all}\ \mathbf{y}\in L\}$ be its dual lattice.
It is shown by Makai in \cite{Ma78} that $L + \KK$ is non-separable if and only if $L^* + \frac{1}{4} \KK^{\circ}$ is a packing. The latter fact is equivalent to saying that the length of a shortest non-zero vector of $L^*$, measured in the norm of $\KK^{\circ}$, is at least $\frac{1}{2}$, or in other words, $\lambda_1(L^*, \KK^{\circ}) \geq \frac{1}{2}$, where $\lambda_1(L^*, \KK^{\circ})$ denotes the first successive minimum of $L^*$ with respect to $\KK^{\circ}$ (\cite{Banaszczyk2}). On the other hand, if $\mu(L, \KK)$ denotes the covering radius of the lattice arrangement $L + \KK$, i.e., the smallest value $\mu$ such that $L + \mu \KK$ covers $\Ee^d$, then by the fact that $(L^*)^*=L$, we have
\[
\lambda_1(L^*, \KK^{\circ}) \cdot \mu(L, \KK) \leq C'd(1+\log d)
\]
for some universal constant $C'>0$ independent of $\KK$, $L$ and $d$.
This inequality was proved by Banaszczyk in \cite[Corollary 1]{Banaszczyk2}. Combining this inequality with the previous observations it follows that $T_{ns}(d) \leq C d(1+\log d)$ for some universal constant $C >0$.

To prove the lower bound we observe that if $L:= \mathbb{Z}^d$ and $$\KK := \frac{1}{2} \conv \{ \pm \e_1, \pm \e_2, \ldots, \pm \e_d \},$$ where the vectors $\e_1, \ldots, \e_d$ are the standard basis vectors, then $L+ \KK$ contains all coordinate axes, and thus it is an NS-lattice arrangement. We further note that the smallest value of $\mu$ such that $L + \mu \KK$ is a covering of $\Ee^d$ is clearly $d$. This completes the proof of Therem~\ref{tightness estimates}.

\section{Appendix: Weakly $k$-impassable lattice arrangements}

As in the case of non-separability, it is a natural to extend the definition of $k$-impassable families for lattice arrangements. In the following, we say that a lattice arrangement of translates of a convex polytope $\mathbf{P}$ is \emph{weakly $k$-impassable} if every $k$-dimensional affine subspace parallel to a facet of $\mathbf{P}$ intersects a member of the arrangement. If $k=d-1$, we may call the arrangement \emph{weakly non-separable}.
Following the paper \cite{KaLo} of Kannan and Lov\'asz, one can define covering minima type quantities for weakly $k$-impassable lattice arrangements of a $d$-dimensional convex polytope for $0\leq k\leq d-1$.

\begin{defn}\label{defn:WNScovmin}
Let $\mathbf{P}$ be a $d$-dimensional convex polytope in $\Ee^d$, and let $L \in \mathcal{L}^d$ be a lattice in $\Ee^d$. For $i=1,\ldots, d$ we define the \emph{$i$th weak covering minimum} of $\mathbf{P}$ with respect to $L$ as the quantity
\[
\mu_i^w(\mathbf{P},L) := \inf \{ t > 0 \ |\ 
\]
\[
L + t\mathbf{P} \hbox{ intersects every } (d-i)-\hbox{dimensional affine subspace parallel to a facet of } \mathbf{P} \}.
\]
\end{defn}

Note that we have $\mu_i^w(\mathbf{P},L) \leq \mu_j^w(\mathbf{P},L)$ for every convex polytope $\mathbf{P}$, lattice $L$, and $1 \leq i \leq j \leq d$, and that $\mu_d^w(\mathbf{P},L)$ coincides with the covering radius of the arrangement.

It is a reasonable question to ask how small the density of a weakly $k$-impassable lattice arrangement of translates of a convex polytope can be. The following result shows that, unlike for non-separable arrangements, for weakly non-separable arrangements it can be arbitrarily small.

\begin{thm}\label{thm:WNSnolattice}
Let $d \geq 2$. Then there is an $\o$-symmetric $d$-dimensional convex polytope $\mathbf{P} \subset \Ee^d$ such that $\mu_1^w(\mathbf{P},\mathbb{Z}^d) = 0$.
\end{thm}

For the proof we need the following form of the well-known theorem of Kronecker from 1884 \cite[Theorem 442]{HW}.

\begin{thm}[Kronecker, 1884]\label{thm:Kronecker}
If $\theta_1, \theta_2, \ldots, \theta_m$ are linearly independent irrational numbers over the field $\mathbb{Q}$, $\alpha_1, \alpha_2, \ldots, \alpha_m \in \Re$, and $N$ and $\varepsilon$ are positive, then there are integers $n>N, p_1,p_2, \ldots, p_m$ such that
\[
\left| n \theta_i - p_i - \alpha_i \right| < \varepsilon
\]
for $i=1,2,\ldots,m$.
\end{thm}

Theorem~\ref{thm:Kronecker} readily implies the following.

\begin{cor}\label{cor:Kronecker}
Let $\mathbf{v} = (\theta_1,\theta_2,\ldots, \theta_d) \in \Ee^d$ such that the coordinates of $\mathbf{v}$ are linearly independent over $\mathbb{Q}$. Let $L$ be the line $L= \mathbb{R} \mathbf{v}$. Then the set $\mathbb{Z}^d + L$ is everywhere dense in $\Ee^d$.
\end{cor}

\begin{proof}
By Theorem~\ref{thm:Kronecker}, the set $\mathbb{Z} \mathbf{v} + \mathbb{Z}^d$ is everywhere dense in $\Ee^d$, immediately implying Corollary~\ref{cor:Kronecker}.
\end{proof}

Now we are ready to prove Theorem~\ref{thm:WNSnolattice}.

\begin{proof}
Choose linearly independent vectors $\mathbf{v}_1, \mathbf{v}_2,\ldots, \mathbf{v}_d$ in $\Ee^d$ such that for every value of $i$, the coordinates of $\mathbf{v}_i$ are linearly independent over $\mathbb{Q}$. Let $\mathbf{P}$ be an $\mathbf{o}$-symmmetric $d$-dimensional parallelotope whose edges are parallel to the $\mathbf{v}_i$. Then, by Corollary~\ref{cor:Kronecker}, the family
\[
\mathcal{P}_{\varepsilon} = \{ \mathbf{n} + \varepsilon \mathbf{P} \ | \ \mathbf{n} \in \mathbb{Z}^d \}
\]
intersects every line parallel to an edge of $\mathbf{P}$. Thus, it intersects every hyperplane parallel to a facet of $\mathbf{P}$, yielding that it is weakly non-separable.
\end{proof}

Our last result shows that a similar statement does not hold for the $i$th weak covering minima of any $d$-dimensional convex polytope $\mathbf{P}$ for $2\leq i\leq d$.
 
\begin{thm}
Let $d \geq 2$, $\mathbf{P}$ be a $d$-dimensional convex polytope in $\Ee^d$ and $L \in \mathcal{L}^d$ be a lattice. Then $\mu_2^w (\mathbf{P}, L) > 0$. Furthermore, if $\mathcal{P}=L+\mathbf{P}$ is a weakly $(d-2)$-impassable lattice arrangement of translates of $\mathbf{P}$, then its density satisfies
\[
\delta(\mathcal{P}) \geq \frac{\vol(\mathbf{P})\vol(\mathbf{P}^{\circ})}{16 \delta(\mathbf{P}^{\circ})},
\]
where $\delta(\mathbf{P}^{\circ})$ denotes the maximum density of a lattice packing of the polar $\mathbf{P}^{\circ}$ of $\mathbf{P}$.
\end{thm}

\begin{proof}
To prove the first statement, we prove that every weakly $(d-2)$-impassable lattice arrangement of a $d$-dimensional convex polytope $\mathbf{P}$ is non-separable; note that this yields $\mu_2^w (\mathbf{P}, L) \geq \mu_1 (\mathbf{P}, L)>0$ for every lattice $L \in \mathcal{L}^d$, where $\mu_1 (\mathbf{P}, L)$ is the first covering minimum of the lattice arrangement (see \cite{KaLo}).

To do it, assume that the arrangement $\mathcal{P}= \{ v + \mathbf{P} | v \in L \}$ is weakly $(d-2)$-impassable for some lattice $L$. Consider an arbitrary hyperplane $H$ of $\Ee^d$. If $H$ is parallel to a facet of $\mathbf{P}$, then by our assumptions $H$ intersects an element of $\mathcal{P}$. Thus, we may assume that $H$ is not parallel to any facet of $\mathbf{P}$. Then, if any $H'$ is any hyperplane parallel to a facet of $\mathbf{P}$, $A=H \cap H'$ is a $(d-2)$-dimensional affine subspace parallel to a facet of $\mathbf{P}$. By our assumptions, $A$ intersects an element of $\mathcal{P}$, showing that $\mathcal{P}$ is an NS-family.

To prove the second part, we apply Theorem 1 of Makai \cite{Ma78} to $\mathbf{P}$ in a straightforward way.
\end{proof}

\bigskip

\noindent K\'aroly Bezdek \\
\small{Department of Mathematics and Statistics, University of Calgary, Canada}\\
\small{Department of Mathematics, University of Pannonia, Veszpr\'em, Hungary\\
\small{E-mail: \texttt{kbezdek@ucalgary.ca}}

\medskip

\noindent Zsolt L\'angi \\
\small{Bolyai Institute, University of Szeged, Szeged, Hungary}\\ 
\small{HUN-REN Alfr\'ed R\'enyi Institute of Mathematics}\\ 
\small{\texttt{zlangi@server.math.u-szeged.hu}}

\end{document}